\documentclass[a4paper,11pt,numreferences,mathsec,kaplist]{isueps}
\usepackage{isu}

\begin{document}
\setcounter{aqwe}{2} 
\begin{article}
\begin{opening}
\title{Irregular linear systems of PDEs with the conditions on solutions' projections }
\author{N. A. ~\surname{Sidorov}}
\institute{Irkutsk State University}

\runningtitle{Irregular linear systems of PDEs with the conditions on solutions' projections}
\runningauthor{ N. A. Sidorov}

\begin{abstract}
The theory of complete generalized Jordan sets is employed to reduce the PDE  with the irreversible linear operator  $B$ of finite index to the regular problems. It is demonstrated how the question of the choice of boundary conditions is connected with the $B$-Jordan structure of coefficients of PDE. The various approaches shows the combination alternative Lyapunov method, Jordan structure coefficients and skeleton decomposition of  irreversible linear operator from the main part equation are among the most powerfull methods to attack such challenging problem. On this base the complex problem of the {\it correct choice of boundary conditions} for the wide class of the singular PDE can be solved. Aggregated existence and uniqueness theorems can be proved, solution may continuously depend on the function determied from the experiments. Such theory can be applied to the integral--differential equations with partial derivatives.
\end{abstract}


\end{opening}

\avtogl{ N. A. Sidorov}{Irregular linear systems of PDEs with conditions on solutions's projections
}
\avtogle{N. Sidorov}{Irregular linear systems of PDEs with conditions on solutions's projections}


\hspace*{3.7cm}{\it This paper is dedicated to the 100th anniversary~of} \\
\hspace*{8.5cm}{\it Irkutsk State University}

\section{Introduction}

Let  $B$ and $A_i$, $i=\overline{1,q}$- are closed linear operators acting from $E_1$ to $E_2$ with the dense domains in $E_1$. $E_1, E_2$ are Banach spaces, $D(B)\subseteq D(A_i),$ $B$ is the operator of finite index with closed range of values, $dim N(B)=n,$ $dim N (B^{\star})=m,$ $\nu=n-m<\infty .$  
Oprerator
$$
L_i\biggl( \frac{\partial}{\partial x}\biggr)=\sum_{|k|\leq q_i}a^i_k(x)D^k, \, \, \, i=1,\dots ,q
$$
be a partial differential operator of order $q_i$, $$q_0>q_1>q_2>\dots >q_q.$$

Functions $a^i_k(x):\Omega \in \mathbb{R}
^N \rightarrow \mathbb{R}^1$ and $f(x):\Omega \in \mathbb{R}^N\rightarrow E_2$ are sufficiently smooth. 
 Differential equation
 \begin{equation}
 L_0\bigl (\frac{\partial}{\partial x}\bigr)Bu+L_1\bigl (\frac{\partial}{\partial x}\bigr)A_1u+\dots +L_q\bigl(\frac{\partial}{\partial x}\bigr)A_qu=f(x)
 \label{eq1}
 \end{equation}
is considered below.
\begin{definition}
 Equation  \eqref{eq1} is a regular equation, if the operator $B$ has bounded inverse. In contrary case we shall say that \eqref{eq1} is a singular equation.
 \end{definition}
   The investigation of singular equation \eqref{eq1} with irreversible operator $B$ in the main part is reduced to the regular problems. The special decomposition of the Banach spaces $E_1$ and $E_2$ in accordance with the generalized $B$-Jordan structure of the operator coefficients $A_1,\dots , A_q$ is used. This reduction makes it possible to formulate the boundary value problems for singular equation \eqref{eq1} with natural conditions on special projections of solution.
   
    Methods of this paper are based on the functional approaches from \cite{lit1,lit2, lit3, lit4, lit5} (here readers may also refer to the monographs \cite{lit6, lit7, lit8, lit10, litdsid}  and extensive bibliographical review of papers of Boris Loginov's school \cite{lit9}).
    
\section{ Preliminaries: $P_k, \, \,  Q_k$ - commutability of linear operators in the case of Noetherian operator $B$}

   Let $E_1=M_1\oplus N_1, \, \, E_2=M_2\oplus N_2,$ $P$ - projector on $M_1$ along $N_1 , Q$  is projector on $M_2$ along $N_2$, A is linear closed operator from $E_1$ to $E_2$, $\overline{D}(A)=E_1,$ $A\in\{A_1,\dots,A_q\}.$
    
\begin{definition}
If $u\in D(A),$   $APu\in D(A),$ $Pu=QAu$, then $A$ be $(P,Q)$ -commute. Let $\{\varphi^{(\prime)}_1,\dots,\varphi^{(\prime)}_n\}$  is a basis in $N(B),$  $\{ \psi^{(\prime)}_1,\dots ,\psi^{(\prime)}_m\}$ is a basis in $N(B^{\star})$
\end{definition}      
 Suppose the following condition is satisfied:
 
      1. Noetherian operator $B$ has a complete $A_1$  -Jordan set $\phi_i^{(j)}$, $B^{\star}$, $i=\overline{1,n}$, $j=\overline{1,p_i}$
  has a complete  $A_1^{\star}$ -Jordan set $\psi_i^{(j)}$, $i=\overline{1,m}$, $j=\overline{1,p_i}$
and the systems $\gamma_i^{(j)}\equiv A_i^{\star}\psi_i^{(p_i+1-j)}$, $ z_i^{(j)}\equiv A_1\phi_1^{(p_i+1-j)}$,
 where $ i=\overline{1,l}, j=\overline{1,p_i}$, $l=\min(m,n)$ corresponding to them are biortogonal \cite{lit1}.
 
The projectors
\begin{equation}
P_k=\sum_{i=1}^l
\sum_{j=1}^{p_i}<\cdot, \gamma_i^{(j)}>\phi_i^{(j)} \stackrel{\mathrm{def}}{\equiv}(<\cdot, \gamma>, \Phi),
\label{eq2} 
\end{equation}
\begin{equation}
Q_k=\sum_{i=1}^l
\sum_{j=1}^{p_i}<\cdot, \psi_i^{(j)}>z_i^{(j)}\stackrel{\mathrm{def}}{\equiv}
(<\cdot, \Psi>, Z),
\label{eq3} 
\end{equation}
where $k=p_1+\dots +p_l$ -root number, generate the direct decomposition
$$
E_1=E_{1k}\oplus E_{1\infty -k,} \, \, \, E_2=E_{2k}\oplus E_{2\infty-k}
$$

\begin{corollary}
 The bounded pseudoinverse operator  $B^{+}$ be $(P_k, Q_k)$ -commute, operator $AB^+$ be $(Q_k,Q_k)$- commute, $B^+A$ be $(P_k,P_k)$- commute, $E_{1\infty-k}, E_{1,k}$  -- invariant subspaces of operator $B^+A,$ $E_{2\infty-k}, E_{2k}$ -- invariant subspaces of operator $AB^+$.
 \end{corollary}
 
 Suppose the operator $A$ be $(P_k,Q_k)$ - commute, where $P_k, Q_k$ are defined by formulas (\ref{eq2}), (\ref{eq3}). Then there is a matrix $\mathcal{A}$, such that $A\Phi=\mathcal{A}Z,$ $A^{\star} \Psi=\mathcal{A}^{\prime}\gamma$. This matrix is called the matrix of $(P_k,Q_k)$ -commutability.
\begin{corollary} 
 The operators $B$ and $A$ be $(P_k,Q_k)$-commute and the matrices of $(P_k,Q_k)$-commutability are the symmetrical cell-diagonal matrices:
$$
\mathcal{A}_B=diag(B_1,\cdots , B_l)  \, \, \, \mathcal{A}_A=diag(\mathcal{A}_1, \cdots , \mathcal{A}_l),
$$   
    where
   $$ B_i=\left[\begin{array}{ccccc}
0&.&.&.&0\\
0&.&.&.&1\\
.&.&.&.&.\\
0&1&.&.&0
\end{array}\right] \, \, \, \, \mathcal{A}_i=\left[\begin{array}{ccc}
0&\cdots&1\\
&\cdots&\\
1&\cdots&0
\end{array}\right] $$ 
 The detailed proof see in preprint \cite{lit4} .
 \end{corollary}   
   
\begin{definition}   
The operator $G$ be $(P_k,Q_k)$-commute quasitriangular, if $\mathcal{A}_G$ is upper quasitriangular matrix, whose diagonal blocks $\mathcal{A}_{ii}$ of dimension $p_i\times p_i$ are lower right triangular matrices.
\end{definition}   
  
\section{The reduction of the equation (\ref{eq1}) to the regular PDE}
Suppose:

2. The operators $A_2,\cdots ,A_q$ be $(P_k,Q_k)$-commute. Then there are matrices $\mathcal{A}_i$ $i=\overline{1,q},$ such that          
 $$
      A_i\Phi=\mathcal{A}_iZ, \, \, \, A_i^{\star}\Psi=\mathcal{A}_i^{\prime}\gamma
 $$ 
  and also $\mathcal{A}_i=(\mathcal{A}_{11},\cdots , \mathcal{A}_{l1})$  is a cell-diagonal matrix ,where              
$$\mathcal{A}_{i1}=\left[\begin{array}{ccc}
0&\cdots&1\\
&\cdots&\\
1&\cdots&0
\end{array}\right], \, \, \, i=\overline{1,l} .$$       
      
 Let us consider the case $m\leq n.$
   
 We introduce the projections $P_k,Q_k$ following formulas  (\ref{eq2}), (\ref{eq3}) and projector
$$      
P_{n-m}=\sum_{i=m+1}^n<\cdot, \gamma_i>\phi_i      
$$  
 generating the direct decompositions
 $$
 E_1=E_{1k}\oplus span\{\phi_{m+1},\cdots , \phi_n\}\oplus E_{1\infty-(k+n-m)},E_2=E_{2k}\oplus E_{2\infty-k}.
 $$     
 Note that $B^+:E_{2\infty-k}\rightarrow E_{1\infty-(k+n-m)}\subset E_{1\infty-k}$, $B^+:E_{2k}\rightarrow E_{1k+n-m}.$
            
We shall seek the solution of the equation (\ref{eq1}) in the following form
\begin{equation}
u(x)=B^+v(x)+(C(x),\Phi)+\sum_{i=m+1}^n\lambda_i(x)\phi_i,
\label{eq4}
\end{equation}      
where $B^+$ is a bounded pseudoinverse operator for $B,$ $v\in E_{2\infty-k},$ $C(x)=(C_1(x),\cdots ,C_m(x))^{\prime},$ $C_i(x)=(C_{i1}(x),\cdots , C_{ip}(x)),$ $\Phi=(\Phi_1,\cdots,\Phi_m)^{\prime},$ $\Phi_i=(\phi_i^1,\cdots , \phi_i^{(p_i)}),$ $i=\overline{1,m}.$

Substituting the expression (\ref{eq4}) into the equation (\ref{eq1}) and noting that $BB^+v=v,$ because $v\in E_{2\infty-k}\subset E_{2\infty-k}\subset E_{2\infty-m}$ we obtain
    
 \begin{equation}
 L_0(\frac{\partial}{\partial x})v+\sum_{i=1}^qL_i(\frac{\partial}{\partial x})A_iB^+v+L_0(\frac{\partial}{\partial x})B(C, \Phi)+\sum_{i=1}^qL_i(\frac{\partial}{\partial x})A_i(C,\Phi)+
 \label{eq5}
 \end{equation}   
$$
+\sum_{j=1}^q\sum_{i=m+1}^nL_j(\frac{\partial}{\partial x})A_j\phi_i\lambda_i(x)=f(x).
$$    
    
 The operator $B^+$ be $(Q_k,P_k)$-  commute, so from the condition 2 and corollary 1 it follows that $Q_kA_iB^+(I-Q_k)=0,$ $(I-Q_k)A_iB^+Q_k=0.$ Hence, $Q_kA_iB^+v=0,$ $\forall v \in E_{2\infty-k}.$
According to the corollary 2 $B\Phi=\mathcal{A}_BZ,$ where $\mathcal{A}_B=(B_1,\cdots,B_m)$ - symmetrical  cell-diagonal matrix. Consequently,
\begin{equation}
(I-Q_k)B\Phi=0, \, \, \, (I-Q_k)A_i\Phi=0, \, \, \, i=\overline{1,q},
\label{eq6}
\end{equation}
because $(I-Q_k)Z=0.$ The following equalities are fulfilled:
\begin{equation}
(A_i(C,\Phi),\Psi)\stackrel{\mathrm{def}}{\equiv} A_i^{\prime}C, \, \, \, (B(C,\Phi), \Psi)\stackrel{\mathrm{def}}{\equiv} \mathcal{A}_BC.
\label{eq7}
\end{equation}   
   
Projecting the equation (\ref{eq5}) onto $E_{2\infty-k}$ by virtue of (\ref{eq6})  we obtain the regular PDE
  \begin{equation}
  \tilde{\mathcal{L}}v=(I-Q_k)f(x)-\sum_{j=1}^q\sum_{i=m+1}^nL_j(\frac{\partial}{\partial x})A_j\phi_i\lambda_i(x),
  \label{eq8}
  \end{equation} 
where
 \begin{equation}
  \tilde{\mathcal{L}}=L_0(\frac{\partial}{\partial x})+\sum_{i=1}^qL_i(\frac{\partial}{\partial x})A_iB^+.
  \label{eq9}
  \end{equation}    
   is the regular differential operator with order $q_0$.
   In order to determine the vector-function $C(x):R^N\rightarrow R^k$, we project the equation (\ref{eq5}) onto $E_{2k}$ and by virtue (\ref{eq7}) we obtain PDE-system
\begin{equation}
 L_0(\frac{\partial}{\partial x})\mathcal{A}_BC+\sum_{i=1}^qL_i(\frac{\partial}{\partial x})\mathcal{A}_i^{\prime}C=<f(x)-\sum_{j=2}^q\sum_{i=m+1}^nL_j(\frac{\partial}{\partial x})A_j\phi_i\lambda_i(x),\Psi>.
  \label{eq10}
  \end{equation}    
So it is proved 
\begin{theorem}
Suppose conditions 1 and 2 are satisfied, $m\leq n, \, \, \, f:\Omega \subset R^N\rightarrow E_2$- sufficiently smooth function. Then any solution of equation (\ref{eq1}) can be represented in the form
$$
u=B^+v+(C,\Phi)+\sum_{i=m+1}^n\lambda_i\phi_i,
$$   
where $v$ satisfies the regular equation (\ref{eq8}), and the vector $C(x)$ is defined from the system (\ref{eq10}). The functions $\lambda_i(x),$ $i=\overline{m+1,n}$ remain an arbitrary functions.
\end{theorem}   
The theorem 1 admits generalizations. Suppose the operators $A_2(x),\cdots,A_q(x)$ with the domains independent from $x$, are subject to the operators $B$ and for any $x\in\Omega$ satisfy to the condition 2. Then the theorem 1 remain valid.  
       
Let us consider system (\ref{eq10}) with unknown vecto-function $C(x)$.    
\begin{lemma}
 Suppose conditions 1,2 are satisfied, operators $A_i,$ $i=\overline{1,q},$ $(P_k,Q_k)$ - commute quasitriangular. Then  the system (\ref{eq10}) is a recurrent sequence of linear differential equations of order $q_1$ with the  regular differential operators of the form
$$
\tilde{\mathcal{L}}_{ks}=L_1(\frac{\partial}{\partial x})+\sum_{i=2}^qa_{pk-s+1,s}^{ik}L_i(\frac{\partial}{\partial x}).
$$              
In particular, if condition 1 is satisfied and $A_2=\cdots=A_q=0,$ system (\ref{eq10}) takes the form
$$
L_1(\frac{\partial}{\partial x})C_{ip_i}(x)=<f(x),\psi_1^{(1)}>,
$$                 
$$
L_1(\frac{\partial}{\partial x})C_{ip_i-s}(x)=<f(x),\psi_i^{(s+1)}>-L_0(\frac{\partial}{\partial x})C_{ip_i-s+1}(x),\, \, \, s=\overline{1,p_i-1}, \, \, \, i=\overline{1,m}.
$$                 
\end{lemma}
              
\begin{corollary}                 
 Let equation (\ref{eq1}) has the form      
$$
L_0(\frac{\partial}{\partial x})Bu+A_1u=f(x)
$$                 
 and condition 1 is satisfied.  Then  the vector $c(x)$  is defiend by simple recursion. 
\end{corollary}
\begin{proof}
 Proof obviously, because in this case in equation (\ref{eq1}) $L_i(\frac{\partial}{\partial x})=1$ and $A_2=\cdots =A_q=0.$

 Let us consider the second case $m>n$.
  
In this case we use the direct decompositions:
$$                   
 E= E_{1k}\oplus E_{1\infty-k}, E_2=E_{2k}\oplus span( z_{n+1},\cdots, z_m )\oplus E_{2\infty-(k+m-n)}               
$$
 and also 
 $B^+:E_{2\infty-(k+m-n)}\rightarrow E_{1\infty-k},B^+:E_{2k+m-n}\rightarrow E_{1k}.$                                                                   
 We shall seek the solution of equation (\ref{eq1}) in the following form
 \begin{equation}
 u(x)=B^+v(x)+(C(x), \Phi),
 \label{eq11}
 \end{equation}                  
 where $v\in E_{2\infty k},C(x)=(C_1(x),\cdots,C_n(x))^{\prime},$ $C_i(x)=(C_{i1}(x), \cdots,C{ip_i}(x),)$
  $$
  \Phi=(\Phi_1,\cdots, \Phi_n)^{\prime}, \, \, \, \Phi_i=(\phi_i^{(1)},\cdots,\phi_i^{(p_i)}), \, \, \, i=\overline{1,n}.
  $$                 
 Substituting (\ref{eq11}) in the equation (\ref{eq1}) we obtain
\begin{equation}     
L_0(\frac{\partial}{\partial x})(I-Q_{m-n})v+\sum_{i=1}^qL_i(\frac{\partial}{\partial x})A_iB^+v+L_0(\frac{\partial}{\partial x})B(C,\Phi)+              
 \label{eq12}
\end{equation}         
$$
+\sum_{i=1}^qL_i(\frac{\partial}{\partial x})A_i(C,\Phi)=f(x)
$$         
                   
with the condition $<v,\psi_i>=0, \, \, \, i=\overline{1,n}, $
 Let the condition 2 is satisfied. Projecting the equation $(\ref{eq12})$  onto the subspaces $E_{2k},\, \, \, E_{2m-n}, \, \, \, E_{2\infty-(k+m-n)}$
 we obtain
\begin{equation}
 L_0(\frac{\partial}{\partial x})\mathcal{A}_BC+\sum_{i=1}^qL_i(\frac{\partial}{\partial x})\mathcal{A}_i^{\prime}C=<f(x),\psi>,            
\label{eq13}
\end{equation} 
\begin{equation}
 \sum_{i=1}^qL_i(\frac{\partial}{\partial x})Q_{m-n}A_iB^+v=Q_{m-n}f(x),            
\label{eq14}
\end{equation}              
\begin{equation}
 L_0(\frac{\partial}{\partial x})(I-Q_{m-n})v+\sum_{i=1}^qL_i(\frac{\partial}{\partial x})Q_{m-n}A_iB^+v=(I-Q_k-Q_{m=n})f(x),            
\label{eq15}
\end{equation}              
 where $v=E_{2\infty-k}, \Psi=(\psi_1,\cdots,\psi_n)^{\prime}, \psi_i \stackrel{\mathrm{def}}{\equiv} (\psi_i^{(1)},\cdots , \psi_i^{(p_i)}).$
The element $v$ can be find from the regular equation
\begin{equation}
\tilde{\mathcal{ L}}v=(I-Q_k)f        
\label{eq16}
\end{equation}               
in the subspace $E_{2\infty-k}\cap E_{2\infty-(m-n)}$. Indeed, if $Q_{m-n}v=0$, then by virtue $Q_{m-n}Q_k=0$ the solution $v$  of the equation $\ref{eq16}$  satisfies to equations (\ref{eq14}), (\ref{eq15}).       
\end{proof}
In this condition we obtain the following result
\begin{theorem}
Let  $n<m$ , conditions  1, 2 are satisfied $f:\Omega\subset R^N\rightarrow E_2$ be sufficiently smooth function. Then any solution of the equation (\ref{eq1}) can be represented in the form (\ref{eq11}), where $v\in E_{2\infty-k}\cap E_{2\infty-(m-n)}$  is the solution of equation (\ref{eq16}) , vector $C$ is defined from the system (\ref{eq13}).         
\end{theorem}

\section{Examples}
 Let operator $B$ be Fredholm $(m=n)$. Then  on the basis of theorems 1,  2 the problem of the choice of correct boundary conditions for  equations (\ref{eq5}), (\ref{eq10}) and (\ref{eq16}), (\ref{eq13}) can be solved for series of concrete differential operators $L_0(\frac{\partial}{\partial x})$ and $L_1(\frac{\partial}{\partial x}).$
 
\textbf{Example 1}

 Consider the equation
 \begin{equation}
 \frac{\partial^2}{\partial x\partial y}B u(x,y)+Au(x,y)=f(x,y)
 \label{eq17}
 \end{equation}
 This equation with usual Goursat conditions $u|_{x=0}=0, \, \, \, u|_{y=0}=0$
  and arbitrary right part evidently has not classical solution.
  
  Let operators $B$ and $A$ satisfy to condition 1, $k=p_1+\cdots+p_n,$ $p_i$
  is a lengths of $A$-Jordan chains of operator $B$

   Then according our theory we can put such conditions on projections of solution:
   \begin{equation}
  (I-P_k)u|_{x=0}=0, \, \, \, (I-P_k)u(x,y)|_{y=0}=0
  \label{eq18}
  \end{equation}
   
   As result we can to construct the following  unique classical solution
$$
u(x,y)=\int_0^x\int_0^y\Gamma\sum_{r=0}^{\infty}(-1)^r(A\Gamma)^r\frac{(x_1-x)^r}{r!}\frac{(y_1-y)^r}{r!}(I-Q_k)f(x_1,y_1)dy_1dx_1+
$$   
$$
+\sum_{i=1}^n \sum_{j=1}^{p_i}
 C_{ij}(x,y)\phi_i^{(j)},
$$
   where $\Gamma=(B+\sum_{i=1}^n < \cdot, \gamma_i > z_i)^{-1}$ is the bounded operator (see Scmidt lemma  in \cite{lit1}) 
   Functions $C_{ij}(x,y)$  are defined recursively
$$
C_ip_i(x,y)=\beta_{i1}(x,y),
$$   
$$
C_ip_{i-1}(x,y)=\beta_{i2}(x,y)-\frac{\partial^2}{\partial x \partial y}C_ip_i(x,y)
$$   
$$
C_ip_{i-2}(x,y)=\beta_{i3}(x,y)-\frac{\partial^2}{\partial x \partial y}C_ip_{i-1}(x,y),
$$   
$$
\cdots
$$   
where $\beta_{is}(x,y)=<f(x,y),\psi_i^{(s)}>, \, \, \, i=\overline{1,n}, \, \, \, s=\overline{1,p_i}.$
   
   Evidently our solution of this special Goursa task continuously depends from right part, if $f(x)\in C^{(p)}$, where $p=max(p_1,\cdots , p_n).$

\textbf{Example 2}

Consider the equation
\begin{equation}
\frac{\partial u(t,x)}{\partial t}-3\int_0^1xs\frac{\partial u(t,s)}{\partial t}ds=u(t,x)+f(t,x) 
\label{eq19}
\end{equation}
with condition
\begin{equation}
u(0,x)-3\int_0^1xsu(0,s)ds=0.
\label{eq20}
\end{equation}
According of (\ref{eq11})  we can construct solution as the sum
$u(t,x)=v(t,x)+c(t)x ,$  where $\int_0^1xv(t,x)dx=0,$ $c(t)=-3\int_0^1xf(t,x)dx$
$$
\frac{\partial v}{\partial t}=v+f(t,x)-3\int_0^1xsf(t,s)ds,
$$    
$$
v|_{t=0}=0.
$$    
Therefore, we have the unique solution of example 2
$$
u(t,x)=\int_0^te^{t-z}\biggl((f(z,x)-3\int_0^1xsf(z,s)ds)\biggr)dz-3\int_0^1xsf(t,s)ds.    
$$

\textbf{Example 3}

Consider the integro-differential equation with order  2
\begin{equation}
\frac{\partial^2u(t,x)}{\partial t^2}-3\int_0^1xs\frac{\partial^2 u(t,s)}{\partial t}ds=\frac{\partial u(t,x)}{\partial t}+f(t,x).
\label{eq21}
\end{equation} 
Continuous function $f(t,x)$ is defined  under $x\in[0,1], \, \, \, t \geq 0.$
Cauchy problem with standard  conditions $u|_{t=0}=0$  $$\frac{\partial u}{\partial t}|_{t=0}=0$$ unsolvable  under arbitrary function $f(t,x)$.    

Projector $P=3\int_0^1xs[\cdot]ds$ corresponds to Fredholm operator
$$
 B=I-3\int_0^1xs[\cdot]ds.
$$    
We can use theorem 1.  Therefore introducing special conditions $u|_{t=0}=0,$
 \begin{equation}
(I-P)\frac{\partial u}{\partial t}\biggl|_{t=0} =0
 \label{eq22}
 \end{equation}
 we can construct solution  as the sum
 $$
 u(t,x)=v(t,x)+c(t)x,
 $$ where $Pv=0.$
 Functions $v(t,x)$ and $c(t)$  can be found from two simplest Cauchy  problems
 \begin{equation*}
 \begin{cases}
   \frac{\partial^2 v}{\partial t^2}=\frac{\partial v}{\partial t}+\frac{dc}{dt}x+f(t,x)\\
   v|_{t=0}=0 \, \, \, \frac{\partial v}{\partial t}|_{t=0}=0,
 \end{cases}
\end{equation*}

\begin{equation*}
 \begin{cases}
   \frac{dc}{dt}+Pf=0\\
  c(0)=0.
 \end{cases}
\end{equation*}

Conditions (\ref{eq22}) were induced by our theorem 1
  As  result we easily can to construct  desired classical solution of  the task (\ref{eq21}), (\ref{eq22})
 $$
 u(t,x)=\int_0^t(e^{t-s}-1)f(s,x)ds-3x\int_0^t\int_0^1 e^{t-s}xf(s,x)dxds.
 $$
 
\textbf{Example 4}

Consider the equation
\begin{equation}
\frac{\partial^2}{\partial x^2}Bu(x,y)+\frac{\partial}{\partial y}Au(x,y)=f(x,y)
\label{eq23}
\end{equation}

Let $B$ be a Fredholm opertor, $(\phi_1,\cdots , \phi_n)$ be a basis in $N(B)$, $(\psi_1, \cdots , \psi_n)$ be a basis in $N(B^{\star})$.

Let 
$$
<A\phi_i,\psi_k>=
\begin{cases}
  1, \, \, \, i=k\\
  0, \, \, \, i\neq k,
 \end{cases}
 \, \, \, P=\sum^n_1<\cdot, A^{\star}\psi_i>\phi_i, \, \, \, Q=\sum^n_1<\cdot,\psi_i>A\phi_i.
$$
Then according of theorem 1 we can put conditions on projetcions:
\begin{equation}
(I-P)u|_{x=0}=0, \, \, \, (I-P)\frac{\partial u}{\partial x}|_{x=0}=0, \, \, \, Pu|_{y=0}=0.
\label{eq24}
\end{equation}
As result we have unique classical solution in form of the sum 
\begin{equation}
u(x,y)=\Gamma v(x,y)+\sum_1^n\int_0^y<f(x,y),\psi_i>dy\phi_i,
\label{eq25}
\end{equation}
where
$$
\Gamma=(B+\sum_1^n<\cdot,A^{\star}\psi_i>A\phi_i)^{-1}
$$
is the bounded operator. Function $v(x,y)$ is the unique solution of the regular Cauchy task
$$
\frac{\partial^2 v(x,y)}{\partial x^2}+A\Gamma\frac{\partial v(x,y)}{\partial y}=(I-Q)f(x,y), \, \, \, v|_{x=0}=0, \, \, \, \frac{\partial v}{\partial x}|_{x=0}=0.
$$
If $f(x,y)$ is analytic function, then $$v(x,y)=\sum_{i=2}^{\infty}C_i(y)x^i,$$ where $$C_2(y)=\frac{1}{2!}(I-Q)f(0,y),$$ $$C_3(y)=\frac{1}{3!}(I-Q)\frac{\partial f(x,y)}{\partial x}|_{x=0},$$ $$  C_4(y)=\frac{1}{4!}(I-Q)\frac{\partial^2 f(x,y)}{\partial x^2}|_{x=0}-\frac{1}{12}A\Gamma\frac{dC_2(y)}{dy},$$
$$
\cdots
$$
Therefore, we have the next asymptotic of solution
$$
u(x,y)=\frac{1}{2}x^2\Gamma f(0,0)+(y-\frac{x^2}{2})\sum_1^n<f(0,0), \psi_i>\phi_i+O(y^2+|x|^3).
$$

\textbf{Example 5}

Consider the equation of 5th order
\begin{equation}
\frac{\partial^3}{\partial t^3}\biggl( \frac{\partial^2}{\partial x^2}+1\biggr)u(x,y,t)+\biggl(\frac{\partial^2}{\partial y^2}+\lambda\biggr)u(x,y,t)=f(x,y,t)
\label{eq26} 
\end{equation}

with boundary conditions

\begin{equation}
u|_{x=0}=0, \, \, \, u|_{x=\pi}=0,
\label{eq27}
\end{equation}

\begin{equation}
u|_{y=0}=0, \, \, \, u|_{y=\pi}=0.
\label{eq28}
\end{equation}

Operator $B=\frac{\partial^2}{\partial x^2}+1$ with condition (\ref{eq27})  maps from $C^{\circ (2)}_{[0,1]}$ in $C_{[0,1]}$. 

Let $\lambda\neq n^2$, $f(x,y,t)$ is continuous function in domain $0\leq x\leq1$, $0\leq y\leq1$, $t\geq 0$.

Let's introduce initial conditions 

\begin{equation}
(I-P)\frac{\partial ^i u(x,y,t)}{\partial t^i}\biggr|_{t=0}=0, \, \, \, i=0,1,2,
\label{eq29}
\end{equation}

where $$P=\frac{2}{\pi}
\int^{\pi}_0\sin x \sin s [\cdot]ds$$ is  projector on $Ker B$.
Then from the proof of theorem 1 it  follows that the equation (\ref{eq26}) with the initial and boundary conditions (\ref{eq27}), (\ref{eq28}), (\ref{eq29})  has the unique classic solution.

\section{Conclusion}

In papers \cite{lit2, lit3, lit4} given the general way how to construct set of correct boundary condition for equation (\ref{eq1}).
For example some authors effectively exploited in mathematical modeling of complex problems boundary Showalter-Sidorov conditions. Such conditions  can be obtained as a partial variant of above stated approach.
Individual interest is represented solution of irregular system PDE (\ref{eq1}), when operator B is assumed to enjoy the skeleton decomposition \cite{lit14}.

 In paper \cite{lit14}    the concept of  a skeleton chains of linear operator is introduce. In this situation the problem of solution singular PDE (\ref{eq1}) 
 also can be reduced to regular split systems of equations.
  The corresponding systems also can be solved with respect  taking into account certain initial and boundary conditions. However the effective use of present results for the applications will be in the future. The development and applications of our functional approach for other nonequal linear and nonlinear integral and integro-differentional systems you can see in references and mathematical reviews (for example, see MR3721762, MR1959647, MR 0810400, MR3343641, MR2920089, MR279574, MR2675324, MR3201397, etc.). This research is supported by Irkutsk State University as part of the project ``Singular operator-differential systems of equations and mathematical models with parameters''.

\bigskip

\textbf{Sidorov Nikolay Alexandrovich}, Doctor of Sciences (Physics and Mathematics),Professor,  Fellow Emiritus,
Irkutsk State University,
1, K. Marks St., Irkutsk, 664003
tel.: (3952)242210
\email{sidorovisu@gmail.com}

\end{article}

\end{document}